\documentclass[11pt]{amsart}

\usepackage{lmodern}
\usepackage{microtype}
\usepackage[english]{babel}
\usepackage{mathtools}
\usepackage{hyperref}
\usepackage[hyperpageref]{backref}
\usepackage[msc-links]{amsrefs}

\theoremstyle{plain} 
\newtheorem{theorem}{Theorem}
\newtheorem{lemma}[theorem]{Lemma}
\newtheorem{corollary}[theorem]{Corollary}

\theoremstyle{definition}
\newtheorem*{definition}{Definition}
\newtheorem{example}[theorem]{Example}

\DeclareMathOperator{\mre}{Re} 
\DeclareMathOperator{\mult}{Mult}
\newcommand{\sigmaa}{\sigma_{\mathrm{a}}}

\begin{document} 
\title{Bohr sets for multiplier algebras of complete Pick spaces}
\date{\today} 

\author{Ole Fredrik Brevig} 
\address{Department of Mathematical Sciences, Norwegian University of Science and Technology (NTNU), 7491 Trondheim, Norway} 
\email{ole.brevig@ntnu.no}

\author{Kristian Seip}
\address{Department of Mathematical Sciences, Norwegian University of Science and Technology (NTNU), 7491 Trondheim, Norway} 
\email{kristian.seip@ntnu.no}

\author{Ilya Zlotnikov}
\address{Department of Mathematics, King's College London, Strand, London, WC2R 2LS, United Kingdom} 
\email{ilia.zlotnikov@kcl.ac.uk}

\thanks{Ole Fredrik Brevig was supported by Grant 354537 of the Research Council of Norway. Kristian Seip and Ilya Zlotnikov were supported by Grant 334466 of the Research Council of Norway.}

\begin{abstract}
	We reexamine F. Wiener's classical proof that the Bohr radius of $H^\infty$ is $1/3$. We replace $H^\infty$ by the multiplier algebras of complete Pick spaces and describe the Bohr set for a wide class of such algebras. In particular, we show that the Bohr set of the multiplier algebra of the Drury--Arveson space $H^2_d$ is the closed $\ell^2$-ball of radius $1/3$. Other cases covered by our description of Bohr sets include Dirichlet-type spaces of analytic functions in the unit disc and a class of Hilbert spaces of ordinary Dirichlet series, studied earlier by McCarthy and Shalit. 
\end{abstract}


\maketitle

\section{Introduction}
The majorant series of an analytic function $\varphi(z) = \sum_{j\geq0} a_j z^j$ in the unit disc $\mathbb{D}$ is
\[M\varphi(z)\coloneq \sum_{j=0}^\infty \lvert a_j z^j \rvert .\]
A well-known solution of F.~Wiener to a problem of Bohr \cite{Bohr1914} is that the bound
\[M\varphi(z) \leq \|\varphi\|_{H^\infty}\]
holds for every $\varphi$ in $H^\infty$ if and only if $\lvert z \rvert \leq 1/3$, where as usual $H^\infty$ is the algebra of bounded analytic functions in $\mathbb{D}$ with the supremum norm. Accordingly, we say that the \emph{Bohr set} of $H^\infty$ is $\{z \in \mathbb{D}\,:\, \lvert z \rvert \leq 1/3\}$ and that its \emph{Bohr radius} is $\varrho(H^\infty)=1/3$.

The purpose of this paper is to elucidate the Bohr phenomenon, starting from the facts that $H^\infty$ is the multiplier algebra of the classical Hardy space $H^2$ and that the Szeg\H{o} kernel
\[k(z,w) = \frac{1}{1-\overline{w}z}\]
is the reproducing kernel of $H^2$ in $\mathbb{D}$. Our reexamination of the Bohr--Wiener proof reveals that it applies in the context of complete Pick spaces, with many classical Hilbert spaces of analytic functions in the unit disc and Hilbert spaces of ordinary Dirichlet series appearing as special cases. 

Some preparation is required to accurately state our main result.

\begin{definition}
	Let $X$ be a set. We say that a reproducing kernel Hilbert space $\mathcal{H}_k$ on $X$ with kernel $k$  is \emph{admissible} if there is a countable abelian semigroup $(\Lambda,+)$ with identity $0$ and an orthogonal basis $\{e_\lambda\}_{\lambda \in \Lambda}$ for $\mathcal{H}_k$ satisfying $e_0 \equiv 1$ and $e_\lambda e_\mu = e_{\lambda+\mu}$, and if there is a point $x_0$ in $X$ such that $k(x,x_0) \equiv 1$.
\end{definition}

If we set $b_\lambda = \|e_\lambda\|_{\mathcal{H}_k}^{-2}>0$ and expand the reproducing kernel of $\mathcal{H}_k$ in the orthogonal basis $\{e_\lambda\}_{\lambda \in \Lambda}$, we obtain
\[k(x,y) = \sum_{\lambda \in \Lambda} b_\lambda e_\lambda(x) \overline{e_\lambda(y)},\]
where the series converges absolutely on $X \times X$ by the Cauchy--Schwarz inequality and Parseval's identity. Every $f$ in $\mathcal{H}_k$ enjoys a unique expansion
\[f = \sum_{\lambda \in \Lambda} \widehat{f}(\lambda) e_\lambda,\]
which is absolutely convergent in $X$. By comparing coefficients in the expansion of $k(x,x_0)$, we see that $b_0 = 1$ and that $e_\lambda(x_0) = 0$ for every $\lambda \neq 0$. In particular, we get that $\widehat{f}(0) = f(x_0)$ for every $f$ in $\mathcal{H}_k$.

The multiplier algebra $\mathcal{M}_k$ consists of those functions $\varphi$ on $X$ such that $f \mapsto \varphi f$ defines a bounded linear operator $M_\varphi$ on $\mathcal{H}_k$, and $\|\varphi\|_{\mathcal{M}_k} \coloneq \|M_\varphi\|$. Since $e_0 \equiv 1$ lies in $\mathcal{H}_k$ with $\|e_0\|_{\mathcal{H}_k} = 1$, we have $\varphi = M_\varphi e_0$, so $\mathcal{M}_k$ is contractively contained in $\mathcal{H}_k$ and the coefficients $\widehat{\varphi}(\lambda)$ are well-defined. 

The majorant function of $\varphi$ in $\mathcal{M}_k$ is
\[M\varphi(x) \coloneq \sum_{\lambda \in \Lambda} \lvert \widehat{\varphi}(\lambda) e_\lambda(x) \rvert.\]
By the Cauchy--Schwarz inequality and the bound $\|\varphi\|_{\mathcal{H}_k} \leq \|\varphi\|_{\mathcal{M}_k}$, we have
\[M\varphi(x) \leq \|\varphi\|_{\mathcal{M}_k} \sqrt{k(x,x)}\]
for every $x$ in $X$. Since $k(x,x) \geq b_0 = k(x_0,x_0) = 1$, the \emph{Bohr set} of $\mathcal{M}_k$ is defined by requiring that this inequality holds with $\sqrt{k(x,x)}$ replaced by its smallest possible value $1$ attained at $x=x_0$, that is
\[\mathcal{B}(\mathcal{M}_k) \coloneq \{ x \in X \,:\, M\varphi(x) \leq \|\varphi\|_{\mathcal{M}_k} \text{ for every } \varphi \in \mathcal{M}_k \}.\]

We now add one further assumption on the kernel.  
 
\begin{definition}
	An admissible reproducing kernel Hilbert space $\mathcal{H}_k$ on $X$ has the \emph{diagonal complete Pick property} if $k\not\equiv 1$ and if $k = 1 + \kappa k$ for a kernel of the form
	\[\kappa(x,y) = \sum_{\lambda \neq 0} c_\lambda e_\lambda(x) \overline{e_\lambda(y)}\]
	with $c_\lambda \geq 0$. We call $\kappa$ the \emph{reciprocal kernel} of $k$.
\end{definition}

Note that if $\mathcal{H}_k$ has the diagonal complete Pick property, then $k$ has no zeros, so we have $\kappa = 1-1/k$. The diagonal complete Pick property implies the complete Pick property and is equivalent to it in classical settings.\footnote{See the discussion preceding Corollary~\ref{cor:powerseries} and Corollary~\ref{cor:diriseries} below.} 

Our main result is that the Bohr set of $\mathcal{M}_k$ admits a complete description when $\mathcal{H}_k$ satisfies the above assumptions.

\begin{theorem} \label{thm:main}
	If $\mathcal{H}_k$ is an admissible reproducing kernel Hilbert space on a set $X$ with the diagonal complete Pick property, then
	\[\mathcal{B}(\mathcal{M}_k) = \left\{ x \in X \,:\, k(x,x) \leq \frac{9}{8} \right\}.\]
\end{theorem}

It should be noted that the conclusion of Theorem~\ref{thm:main} can be equivalently formulated via the reciprocal kernel $\kappa$, where it takes the more familiar form
\begin{equation} \label{eq:famform}
	\mathcal{B}(\mathcal{M}_k) = \left\{ x \in X \,:\, \sqrt{\kappa(x,x)} \leq \frac{1}{3} \right\}.
\end{equation}

It is clear that the Hardy space $H^2$ is an admissible reproducing kernel Hilbert space on $\mathbb{D}$ with $e_j(z) = z^j$ for $j=0,1,2,\ldots$. If $k$ is the Szeg\H{o} kernel, then the reciprocal kernel is
\[\kappa(z,w) = e_1(z) \overline{e_1(w)} = z \overline{w},\] 
and we recover the Bohr--Wiener result from \eqref{eq:famform} since $\sqrt{\kappa(z,z)}=\lvert z \rvert$.

Our next example is the Drury--Arveson space $H^2_d$, which will be of central importance in the proof of Theorem~\ref{thm:main}. Fix $d$ in $\mathbb{N} \cup \{\infty\}$ and let $\mathbb{B}_d$ be the unit ball of $\mathbb{C}^d$, where $\mathbb{C}^d$ is understood as $\ell^2$ when $d = \infty$. The Drury--Arveson space $H^2_d$ is the reproducing kernel Hilbert space on $\mathbb{B}_d$ with kernel
\[k_d(z,w) \coloneq \frac{1}{1 - \langle z,w \rangle_d},\]
where $\langle z,w \rangle_d$ is the standard inner product in $\mathbb{C}^d$ or in $\ell^2$. Notice that $H^2_1$ is just the classical Hardy space $H^2$.

To see that $H^2_d$ is admissible, we let $\Lambda$ be the countable abelian semigroup of finitely supported multi-indices $\gamma = (\gamma_j)_{j=1}^d$ of nonnegative integers under addition. The set of monomials $e_\gamma(z) = z^\gamma$ constitutes an orthogonal basis of $H^2_d$, and since $k_d(z,0)=1$, the point $x_0$ is the origin. Since $\kappa(z,w) = \langle z,w\rangle_d$, the space $H^2_d$ also has the diagonal complete Pick property.

Every $F$ in $H^2_d$ enjoys an absolutely convergent monomial expansion $F(z) = \sum_\gamma a_\gamma z^\gamma$ in $\mathbb{B}_d$, and the Bohr set is defined through the corresponding majorants $M\Phi(z) = \sum_\gamma \lvert a_\gamma z^\gamma \rvert$ for $\Phi$ in the multiplier algebra $\mult(H^2_d)$.

\begin{theorem} \label{thm:da}
	For every $d$ in $\mathbb{N} \cup \{\infty\}$ the Bohr set of $\mult(H^2_d)$ equals  
	\[\left\{ z \in \mathbb{B}_d \,:\, \|z\|_{\ell^2} \leq \frac{1}{3} \right\}.\]
\end{theorem}

Theorem~\ref{thm:da} is thus a special case of Theorem~\ref{thm:main}, and convention dictates that such a result be labeled as a corollary. We have not done so because the Drury--Arveson space is, rather, one of the main ingredients in the proof of Theorem~\ref{thm:main}, the relevant mechanism being its \emph{universality}. The proof of Theorem~\ref{thm:main} consists of two parts:

\begin{enumerate}
	\item[(i)] Show that if $k(x,x)>9/8$, then $x$ is not in $\mathcal{B}(\mathcal{M}_k)$. This will be handled by general arguments using the properties of $\kappa$.
	\item[(ii)] Show that if $k(x,x)\leq 9/8$, then $x$ is in $\mathcal{B}(\mathcal{M}_k)$. Here we will first establish the result for the Drury--Arveson space and use its universality (via $\kappa$) to extend the result to $\mathcal{H}_k$.
\end{enumerate}

The Drury--Arveson part of (ii) was first obtained by Paulsen, Popescu, and Singh~\cite{PPS2002}*{Example~5} for $d<\infty$. Their proof is different from the one to be given below, which hews closer to the original Bohr--Wiener approach. Our version of the latter is essentially due to Boas and Khavinson~\cite{BK1997}, who established that the Bohr set of $H^\infty(\mathbb{D}^d)$ contains $\{z \in \mathbb{D}^d\,:\, \|z\|_{\ell^2} \leq 1/3\}$. This should be compared with the Drury--Arveson setting presented in Theorem~\ref{thm:da}, where the novelty lies mainly in the reverse inclusion. In this context one should also recall that $H^2(\mathbb{D}^d)$ is not a complete Pick space unless $d=1$, so Theorem~\ref{thm:main} does not apply.

The assertion in (i) is new for the Drury--Arveson space when $d\geq2$. In fact, to the best of our knowledge, the full Bohr set of any space of functions of several variables has never been computed. 

Our next class of examples arises from power series. We choose again $e_j(z) = z^j$ for $j = 0,1,2,\ldots$, so that $\mathcal{H}_k$ has reproducing kernel
\[k(z,w) = \sum_{j=0}^\infty b_j (\overline{w}z)^j\]
for $b_0 = 1$ and $b_j > 0$. We want to compare different spaces of power series on the same set $X = \mathbb{D}$, and will do this by enforcing the requirement that the radius of convergence of the power series $\sum_{j \geq 0} b_j z^j$ be exactly $1$. We refer to such a space as an \emph{admissible space of power series}. The function $r \mapsto k(r,r)$ is then continuous and strictly increasing on $[0,1)$, with $k(0,0) = 1$.

Let $\varphi$ be in the multiplier algebra $\mathcal{M}_k$. The majorant $M\varphi(z)$ depends only on $\lvert z \rvert$, so the Bohr set of an admissible space of power series is a disc centered at the origin. We define the \emph{Bohr radius} of $\mathcal{M}_k$ as
\[\varrho(\mathcal{M}_k) \coloneq \sup\{\lvert z \rvert \,:\, z \in \mathcal{B}(\mathcal{M}_k)\}.\]

It follows from \cite{AM2002}*{Theorem~7.33} that for admissible spaces of power series the diagonal complete Pick property defined above is no stronger than the usual complete Pick property. The next result identifies the Bohr radius of every admissible space of power series with the complete Pick property, and shows that among these the Hardy space $H^2$ has the smallest one.

\begin{corollary} \label{cor:powerseries}
	Let $\mathcal{H}_k$ be an admissible space of power series with the complete Pick property. If there is a (necessarily unique) $1/3 \leq \varrho < 1$ such that $k(\varrho,\varrho) = 9/8$, then
	\[\mathcal{B}(\mathcal{M}_k) = \{ z \in \mathbb{D} \,:\, \lvert z \rvert \leq \varrho \}.\]
	Otherwise $\mathcal{B}(\mathcal{M}_k) = \mathbb{D}$. Moreover, $\varrho = 1/3$ if and only if $k$ is the Szeg\H{o} kernel.
\end{corollary}

Let us consider two classical scales of admissible spaces of power series.

\begin{example}[Binomial weights] \label{ex:binom}
	For $0<\alpha<\infty$, let $A^2_\alpha$ be the reproducing kernel Hilbert space on $\mathbb{D}$ with reproducing kernel
	\[k_\alpha(z,w) = \frac{1}{(1-\overline{w}z)^\alpha}.\]
	Each $A^2_\alpha$ is plainly an admissible space of power series, but the binomial theorem shows that $k_\alpha$ has the diagonal complete Pick property if and only if $0<\alpha \leq 1$. If $\alpha>1$, then $A^2_\alpha$ is a Bergman-type space and $\mult(A^2_\alpha) = H^\infty$. Thus we have by Corollary~\ref{cor:powerseries} and the Bohr--Wiener theorem that
	\[\varrho_\alpha = \begin{cases}
		\sqrt{1-(8/9)^{1/\alpha}}, & \text{if } 0 < \alpha \leq 1; \\
		1/3, & \text{if } 1 \leq \alpha < \infty.
	\end{cases}\]
\end{example}

\begin{example}[Power weights]
	For $0 < \beta < \infty$, consider the admissible space of power series with reproducing kernel
	\[k_\beta(z,w) = \sum_{n=0}^\infty (n+1)^{-\beta} (\overline{w}z)^n,\]
	which has the diagonal complete Pick property, since $b_n = (n+1)^{-\beta}$ are logarithmically convex and Kaluza's lemma~\cite{AM2002}*{Lemma~7.38} then gives $c_n \geq 0$ for every $n$. Since $k_\beta(r,r)$ increases to $\sum_{n\geq1} n^{-\beta}$ as $r \to 1^-$, we are in the second case of Corollary~\ref{cor:powerseries} precisely when $\beta>1$ and $\zeta(\beta) \leq 9/8$. Hence, if $\beta_0 = 3.51545\ldots$ denotes the solution of $\zeta(\beta_0)=9/8$, then $\mathcal{B}(\mathcal{M}_{k_\beta}) = \mathbb{D}$ for $\beta \geq \beta_0$, while for $0<\beta<\beta_0$ the Bohr radius $\varrho_\beta$ is the solution of $k_\beta(r,r)=9/8$ in $(1/3,1)$.
\end{example}

Our final class of examples arises from ordinary Dirichlet series, following McCarthy and Shalit~\cite{MS2017}. Let $\mathcal{N}\neq\{1\}$ be a sub-semigroup of $(\mathbb{N},\cdot)$ containing $1$, and choose $\Lambda = \{\log n \,:\, n \in \mathcal{N}\}$ and $e_{\log n}(s) = n^{-s}$. The reproducing kernel of $\mathcal{H}_k$ is then
\[k(s,w) = \sum_{n \in \mathcal{N}} b_n n^{-s-\overline{w}}\]
with $b_1 = 1$ and $b_n > 0$ for every $n$ in $\mathcal{N}$. It follows that $\mathcal{H}_k$ consists precisely of those Dirichlet series $f(s) = \sum_{n \in \mathcal{N}} \widehat{f}(n) n^{-s}$ for which $\sum_{n \in \mathcal{N}} \lvert \widehat{f}(n) \rvert^2 /b_n$ is finite. Consequently, every $f$ in $\mathcal{H}_k$ converges wherever $k$ does, and
\[\sigmaa \coloneq \inf\left\{ \sigma \in \mathbb{R} \,:\, \sum_{n \in \mathcal{N}} b_n n^{-2\sigma} < \infty \right\}\]
is such that $\mathbb{C}_{\sigmaa} \coloneq \{s \,:\, \mre{s} > \sigmaa\}$ is the largest half-plane of convergence for $\mathcal{H}_k$. We take $X$ to be the \emph{extended} half-plane $\mathbb{C}_{\sigmaa} \cup \{+\infty\}$ and refer to $\mathcal{H}_k$ as an \emph{admissible space of Dirichlet series}. Notice that the point at $+\infty$ plays the role of $x_0$, and we adopt the convention $\mre{+\infty}=+\infty$.

Since the majorant $M\varphi(s)$ for $\varphi$ in the multiplier algebra $\mathcal{M}_k$ depends only on $\mre{s}$, the Bohr set of an admissible space of Dirichlet series is an extended half-plane contained in $X$. We define the \emph{Bohr abscissa}\footnote{This is the ``isometric Bohr abscissa'' of Balasubramanian, Calado, and Queff\'elec~\cite{BCQ2006}.} of $\mathcal{M}_k$ as
\[\varsigma(\mathcal{M}_k) \coloneq \inf\{\mre{s} \,:\, s \in \mathcal{B}(\mathcal{M}_k)\}.\]

For admissible spaces of Dirichlet series, the diagonal complete Pick property is again no stronger than the complete Pick property, this time by~\cite{MS2017}*{Theorem~26}. Moreover, if $\mathcal{H}_k$ has the complete Pick property, then $\sigmaa > -\infty$ by~\cite{MS2017}*{Theorem~28}, so that $\mathbb{C}_{\sigmaa}$ is a genuine half-plane. A change of variables would allow us to normalize $\sigmaa = 0$, as in~\cite{MS2017}, but we prefer to keep track of $\sigmaa$ explicitly. The next result identifies the Bohr abscissa of every admissible space of Dirichlet series with the complete Pick property, and identifies the kernel with the largest Bohr abscissa relative to $\sigmaa$.

\begin{corollary} \label{cor:diriseries}
	Suppose that $\mathcal{H}_k$ is an admissible space of Dirichlet series that enjoys the complete Pick property. If there is a (necessarily unique) $\sigmaa < \varsigma \leq \sigmaa + \log 3/\log 2$ such that $k(\varsigma,\varsigma) = 9/8$, then
	\[\mathcal{B}(\mathcal{M}_k) = \{ s \in \mathbb{C}_{\sigmaa} \cup \{+\infty\} \,:\, \mre{s} \geq \varsigma \}.\]
	Otherwise $\mathcal{B}(\mathcal{M}_k) = \mathbb{C}_{\sigmaa} \cup \{+\infty\}$. Moreover, $\varsigma = \sigmaa + \log 3/\log 2$ if and only if
	\[k(s,w) = \frac{1}{1 - 2^{2\sigmaa-s-\overline{w}}}.\]
\end{corollary}

Let us look at two examples where Corollary~\ref{cor:diriseries} applies.

\begin{example}[McCarthy's space~\cite{McCarthy2004}]
	Let $\mathcal{H}_k$ be the admissible space of Dirichlet series with kernel
	\[k(s,w) = \frac{1}{2-\zeta(s+\overline{w})},\]
	whose coefficients $b_n$ count the ordered factorizations of $n$. The abscissa $\sigmaa$ is the unique solution of $\zeta(2\sigma) = 2$, and $\mathcal{H}_k$ has the complete Pick property, since $\kappa(s,w) = \zeta(s+\overline{w})-1$. It follows by Corollary~\ref{cor:diriseries} that the Bohr abscissa $\varsigma(\mathcal{M}_k)$ is the unique solution of $\zeta(2\sigma) = 10/9$. Numerically,
	\[\sigmaa = 0.86432\ldots \qquad\text{and}\qquad \varsigma(\mathcal{M}_k) = 1.82451\ldots,\]
	so that the difference $\varsigma(\mathcal{M}_k) - \sigmaa = 0.96018\ldots$ lies comfortably below the universal bound $\log 3/\log 2 = 1.58496\ldots$ from Corollary~\ref{cor:diriseries}.
\end{example}

\begin{example}[A Dirichlet series realization of $H^2_\infty$]
	Let $P(s) = \sum_p p^{-s}$ denote the prime zeta function and consider the kernel
	\[k(s,w) = \frac{1}{1-P(s+\overline{w})},\]
	so that $\kappa(s,w) = P(s+\overline{w})$ and $\mathcal{N} = \mathbb{N}$. Thus $\mathcal{H}_k$ is an admissible space of Dirichlet series with the complete Pick property, and its abscissa $\sigmaa$ and Bohr abscissa $\varsigma(\mathcal{M}_k)$ are, respectively, the solutions of $P(2\sigma)=1$ and $P(2\sigma)=1/9$, numerically
	\[\sigmaa = 0.69971\ldots \qquad\text{and}\qquad \varsigma(\mathcal{M}_k) = 1.77072\ldots.\]
	It is demonstrated in \cite{MS2017}*{Theorem~41} that $\mathcal{H}_k$ is weakly isomorphic to $H^2_\infty$ and $\mathcal{M}_k$ is unitarily equivalent to $\mult(H^2_\infty)$.
\end{example}

\subsection*{Organization} This paper comprises three additional sections. In Section~\ref{sec:pdk} we prove part (i) of Theorem~\ref{thm:main}. Section~\ref{sec:da} concerns the proof of part (ii) and goes via the Drury--Arveson space and Theorem~\ref{thm:da}. In both sections we begin by recalling the Bohr--Wiener proof in the classical setting and explain how it extends to the general setting. The final Section~\ref{sec:fj} contains the proofs of Corollary~\ref{cor:powerseries} and  Corollary~\ref{cor:diriseries}.

\subsection*{AI disclosure} In the preparation of this paper we have been assisted by Claude Fable 5, who played the role of an expert in the theory of complete Pick spaces. We were initially considering only the spaces of Example~\ref{ex:binom} and had found the main ideas from Section~\ref{sec:pdk} and a coarse version of the argument in Section~\ref{sec:da} that circumvented the Drury--Arveson space. Claude was presented with our work and \cites{AM2002,Hartz2023} with the suggestion to use the complete Pick property of the kernel to tighten our upper and lower bounds on $\varrho_\alpha$. After a few iterations, the result stated in Example~\ref{ex:binom} was obtained. The abstract generalization and the application to Dirichlet series were identified by the authors in collaboration with Claude. The authors wrote the final version of all arguments and take complete responsibility for the content of the paper.

\section{Positive semi-definite kernels} \label{sec:pdk}
We motivate our approach by going over the Bohr--Wiener proof that any $z$ in $\mathbb{D}$ with $\lvert z \rvert >1/3$ cannot belong to the Bohr set of $H^\infty$. The idea is that the conformal map
\begin{equation} \label{eq:confmap}
	\varphi_a(z) \coloneq \frac{a-z}{1-\overline{a}z}
\end{equation}
satisfies $\|\varphi_a\|_{H^\infty} = 1$ for every $a$ in $\mathbb{D}$, but expanding
\begin{equation} \label{eq:classicalexp}
	\varphi_a(z) = a - (1-\lvert a \rvert^2) \sum_{m=1}^\infty \overline{a}^{\,m-1} z^m,
\end{equation}
we find that
\[M \varphi_a(z) = \lvert a \rvert + \frac{(1-\lvert a \rvert^2)\lvert z \rvert}{1-\lvert a \rvert \lvert z \rvert}.\]
It follows that
\[M \varphi_a(z) \leq \|\varphi_a\|_{H^\infty} \qquad \iff \qquad \lvert z \rvert \leq \frac{1}{1+2\lvert a \rvert},\]
so letting $\lvert a\rvert \to 1^-$, we obtain the desired conclusion.

The key to generalizing the example \eqref{eq:confmap} is to understand that it is actually constructed from three ingredients. The first is an initial multiplier $\varphi$ of norm $\leq 1$, which in this case is $\varphi(z)=z$. The second is a way to create a family of multipliers $\varphi_a$ from the initial multiplier. In the case of $H^\infty$, it is obvious that if the norm of $\varphi$ is $\leq 1$, then so is the norm of
\[\varphi_a(z) = \frac{a-\varphi(z)}{1-\overline{a}\varphi(z)}.\]
The purpose of the second ingredient is amplification. The initial multiplier $\varphi(z)=z$ carries no information by itself, since $M\varphi(z)=\lvert z\rvert$ does not exceed $1$, whereas letting $\lvert a \rvert \to 1^-$ in the family, we turn this trivial bound into the sharp requirement $\lvert z \rvert \leq 1/3$. The third ingredient is that no cancellation occurs when the majorant of $\varphi_a$ is formed, which for $H^\infty$ may be read off from \eqref{eq:classicalexp}.

In order to generalize the first two of these steps, we must be able to estimate norms in reproducing kernel Hilbert spaces and in their multiplier algebras. The unit balls of $\mathcal{M}_k$ and of $\mathcal{H}_k$ are both described by positivity. A kernel $K \colon X \times X \to \mathbb{C}$ is called \emph{positive semi-definite} if
\[\sum_{m,n=1}^N c_m \overline{c_n} K(x_m,x_n) \geq 0 \]
for every finite collection of points $x_1,x_2,\ldots,x_N$ in $X$ and all complex scalars $c_1,c_2,\ldots,c_N$. The first description is well-known (see \cite{AM2002}*{Corollary~2.37} or \cite{Hartz2023}*{Proposition~10.4.3}).

\begin{lemma} \label{lem:mnormest}
	Suppose that $\mathcal{H}_k$ is a reproducing kernel Hilbert space on $X$. A function $\varphi$ is in the unit ball of $\mathcal{M}_k$ if and only if
	\[K(x,y) = \left(1-\varphi(x)\overline{\varphi(y)}\right) k(x,y)\]
	is positive semi-definite.
\end{lemma}

The description of the unit ball of $\mathcal{H}_k$ follows from Aronszajn's comparison theorems (see~\cite{Aronszajn1950}*{I.\S7}) in which the smaller of the two spaces is the one-dimensional space with reproducing kernel $(x,y) \mapsto f(x)\overline{f(y)}$.

\begin{lemma} \label{lem:domination}
	Suppose that $\mathcal{H}_k$ is a reproducing kernel Hilbert space on $X$. A function $f$ on $X$ belongs to the unit ball of $\mathcal{H}_k$ if and only if the kernel
	\[K(x,y) = k(x,y) - f(x)\overline{f(y)}\]
	is positive semi-definite.
\end{lemma}

We shall also need to know that the pointwise product of two positive semi-definite kernels is again positive semi-definite. This is a well-known consequence of the Schur product theorem (see \cite{AM2002}*{Theorem~A.1}).

We next turn to the adaptation of the second ingredient, which is a general fact concerning reproducing kernel Hilbert spaces. Note that it requires nothing of $\mathcal{H}_k$ beyond $k(x,x)>0$, which holds in any admissible space since $k(x,x)\geq k(x_0,x_0) = 1$.

\begin{lemma} \label{lem:conformal}
	Let $\mathcal{H}_k$ be a reproducing kernel Hilbert space on $X$ such that $k(x,x)>0$ for every $x$ in $X$. If $\|\varphi\|_{\mathcal{M}_k} \leq 1$, then
  	\[\varphi_a(x) = \frac{a-\varphi(x)}{1-\overline{a} \varphi(x)}\]
	also satisfies $\|\varphi_a\|_{\mathcal{M}_k}\leq 1$ for any $a$ in $\mathbb{D}$.
\end{lemma}

\begin{proof}
	By Lemma~\ref{lem:mnormest}, the kernel
	\[K(x,y) = \left(1-\varphi(x) \overline{\varphi(y)}\right) k(x,y)\]
	is positive semi-definite. In particular $K(x,x) \geq 0$, which when combined with the assumption that $k(x,x) > 0 $ implies that $|\varphi(x)|\leq 1$ for every $x$. This means that $\varphi_a$ is well-defined. We next compute
	\[\left(1-\varphi_a(x)\overline{\varphi_a(y)}\right)k(x,y) = \frac{1-|a|^2}{(1-\overline{a}\varphi(x))(1-a\overline{\varphi(y)})} K(x,y).\]
	By Lemma~\ref{lem:mnormest} and the consequence of the Schur product theorem discussed above, it remains to check that
	\[\frac{1-|a|^2}{(1-\overline{a}\varphi(x))(1-a\overline{\varphi(y)})}\]
	is positive semi-definite. However, this is trivial since every kernel of the form $(x,y) \mapsto g(x)\overline{g(y)}$ is positive semi-definite.
\end{proof}

We turn to the generalization of the first ingredient. We now need to require in addition that the kernel $\kappa = 1-1/k$ is positive semi-definite, so that it is itself the reproducing kernel of a Hilbert function space $\mathcal{H}_\kappa$ on $X$. For an admissible space $\mathcal{H}_k$, the kernel $\kappa$ is positive semi-definite if and only if $\mathcal{H}_k$ is a complete Pick space, since $k$ is normalized at $x_0$ (see e.g.~\cite{Hartz2023}*{Corollary~10.4.10}).

\begin{lemma} \label{lem:hkappa}
	Suppose that $\mathcal{H}_k$ is a reproducing kernel Hilbert space on $X$ and that $k = 1 + \kappa k$ for some positive semi-definite kernel $\kappa$ on $X$. If $\varphi$ is in $\mathcal{H}_\kappa$, then $\varphi$ is in $\mathcal{M}_k$ and
	\[\|\varphi\|_{\mathcal{M}_k} \leq \|\varphi\|_{\mathcal{H}_\kappa}.\]
\end{lemma}

\begin{proof}
	We may assume without loss of generality that $\|\varphi\|_{\mathcal{H}_\kappa} = 1$. Since $k = 1 + \kappa k$, we have
	\[\left(1-\varphi(x)\overline{\varphi(y)}\right)k(x,y) = 1 + \left(\kappa(x,y) - \varphi(x)\overline{\varphi(y)}\right)k(x,y).\]
	The kernel $(x,y) \mapsto \kappa(x,y)-\varphi(x)\overline{\varphi(y)}$ is positive semi-definite by Lemma~\ref{lem:domination}, so the right-hand side is positive semi-definite by the consequence of the Schur product theorem discussed above. Lemma~\ref{lem:mnormest} now gives $\|\varphi\|_{\mathcal{M}_k}\leq1$.
\end{proof}

Using Lemma~\ref{lem:hkappa}, we can construct an initial multiplier of norm at most $1$ from each point $\xi$ in $X$ with $k(\xi,\xi)>1$. Such points exist because $k \not\equiv 1$. Our choice of multiplier is the normalized reproducing kernel
\begin{equation} \label{eq:initial}
	\varphi(x) = \frac{\kappa(x,\xi)}{\sqrt{\kappa(\xi,\xi)}}
\end{equation}
of $\mathcal{H}_\kappa$ at $\xi$. For the Szeg\H{o} kernel we have $\kappa(z,w)=\overline{w}z$, so that $\varphi(z) = \overline{\xi}z/\lvert \xi \rvert$. This coincides with the initial multiplier of the classical proof up to a unimodular factor, which has no effect on the argument.

Let us see how the first two ingredients fit together. Since $\|\varphi\|_{\mathcal{M}_k}\leq1$ by construction, we get that $\lvert \varphi(x)\rvert \leq 1$ for every $x$ in $X$ as in the proof of Lemma~\ref{lem:conformal}. Since $\lvert a \rvert<1$, this yields the expansion
\begin{equation} \label{eq:phiaexp}
	\varphi_a(x) = \frac{a-\varphi(x)}{1-\overline{a}\varphi(x)} = a - (1-\lvert a \rvert^2) \sum_{m=1}^\infty \overline{a}^{\,m-1}\varphi^m(x),
\end{equation}
the series being absolutely convergent at every point of $X$. Since $M_{\varphi^m} = M_\varphi^m$, we have $\|\varphi^m\|_{\mathcal{M}_k}\leq1$, so the series converges absolutely in $\mathcal{M}_k$ and \eqref{eq:phiaexp} holds there as well.

Our goal is now to compute $M\varphi_a(\xi)$, which requires the majorants of the powers $\varphi^m$ at $\xi$ and the knowledge that no cancellation occurs when they are assembled via \eqref{eq:phiaexp}. This is the third ingredient in our proof. For the Szeg\H{o} kernel, both these steps are plain, because $\varphi^m$ is a unimodular multiple of $z^m$. In general, $\varphi^m$ is not a multiple of a single $e_\lambda$, and distinct powers may contribute to the same coefficient. We must therefore identify the expansion of $\varphi^m$ viewed as a function in $\mathcal{H}_k$.

\begin{lemma} \label{lem:powers}
	Suppose that $\mathcal{H}_k$ is an admissible reproducing kernel Hilbert space on $X$ with the diagonal complete Pick property and let $m\geq1$. Then
	\begin{equation} \label{eq:kappam}
		\kappa^m(x,y) = \sum_{\lambda \neq 0} c_\lambda(m)\, e_\lambda(x)\overline{e_\lambda(y)}
	\end{equation}
	for coefficients satisfying $0 \leq c_\lambda(m) \leq b_\lambda$. Consequently, if $k(\xi,\xi)>1$ and $\varphi(x) = \kappa(x,\xi)/\sqrt{\kappa(\xi,\xi)}$, then the expansion of $\varphi^m$ in $\mathcal{H}_k$ is
	\begin{equation}\label{eq:varphim}
		\varphi^m = \frac{1}{\kappa(\xi,\xi)^{m/2}} \sum_{\lambda \neq 0} c_\lambda(m) \overline{e_\lambda(\xi)}\, e_\lambda.
	\end{equation}
\end{lemma}
\begin{proof}
	It follows from the Cauchy--Schwarz inequality that the expansion
	\[\kappa(x,y) = \sum_{\lambda \neq 0} c_\lambda e_\lambda(x) \overline{e_\lambda(y)}\]
	converges absolutely in $X \times X$. Multiplying out $\kappa^m$ and rearranging, we therefore obtain the expansion 
	\[\kappa^m(x,y) = \sum_{\lambda \in \Lambda} c_\lambda(m)\, e_\lambda(x)\overline{e_\lambda(y)}\]
	with  
	\[c_\lambda(m) = \sum_{\lambda_1+\cdots+\lambda_m = \lambda} c_{\lambda_1} \cdots c_{\lambda_m} \geq 0.\]
	Notice that the sum defining $c_\lambda(m)$ converges to a finite number, because the rearrangement gives $c_\lambda(m) \lvert e_\lambda(x)\rvert^2 \leq \kappa(x,x)^m$ for every $x$ in $X$ and no $e_\lambda$ vanishes identically since $\{e_\lambda\}_{\lambda \in \Lambda}$ is an orthogonal basis for $\mathcal{H}_k$. Since $\kappa^m(x_0,y) = \kappa(x_0,y)^m = 0$ for every $y$ in $X$, we get $c_0(m)=0$ and hence \eqref{eq:kappam}.
	
	The kernel $\kappa$ is positive semi-definite by the diagonal complete Pick property, and we plainly have $\kappa<1$ on the diagonal since $k(x,x)$ is finite. The Cauchy--Schwarz inequality therefore yields that $\lvert \kappa \rvert<1$, so the series $k = \sum_{j\geq0}\kappa^j$ converges absolutely. Moreover, every power $\kappa^j$ is positive semi-definite by the Schur product theorem. Combining this fact with \eqref{eq:kappam}, we find that the kernel 
	\[k(x,y) - c_\lambda(m)\, e_\lambda(x)\overline{e_\lambda(y)} = \sum_{j \neq m} \kappa^j(x,y) + \left(\kappa^m(x,y) - c_\lambda(m)\, e_\lambda(x)\overline{e_\lambda(y)}\right)\]
	is positive semi-definite. Lemma~\ref{lem:domination} applied to $f=\sqrt{c_\lambda(m)} e_\lambda$ then yields that $c_\lambda(m)\|e_\lambda\|_{\mathcal{H}_k}^2 \leq 1$, which is equivalent to the desired estimate $c_\lambda(m) \leq b_\lambda$.

	We turn to the final assertion. Setting $y=\xi$ in \eqref{eq:kappam}, we see that \eqref{eq:varphim} converges pointwise absolutely in $X$. Since $\{e_\lambda\}_{\lambda\neq0}$ is an orthogonal system in $\mathcal{H}_k$, this is the expansion of $\varphi^m$ in $\mathcal{H}_k$ once we know it defines an element of $\mathcal{H}_k$. Using the bound $c_\lambda(m) \leq b_\lambda$, we find that
	\[\sum_{\lambda\neq0} \frac{\lvert c_\lambda(m)e_\lambda(\xi)\rvert^2}{b_\lambda} \leq \sum_{\lambda\neq0} b_\lambda \lvert e_\lambda(\xi)\rvert^2 = k(\xi,\xi)-1 < \infty,\]
	which completes the proof.
\end{proof}

Neither Lemma~\ref{lem:conformal} nor Lemma~\ref{lem:hkappa} uses admissibility. Lemma~\ref{lem:powers} uses all of it: the semigroup relation and the orthogonality of $\{e_\lambda\}_{\lambda\in\Lambda}$ produce the expansion, and the normalization at $x_0$ keeps the powers of $\kappa$ away from the coefficient at $0$. The diagonal complete Pick property supplies the signs that will preclude cancellation.

\begin{proof}[First part of the proof of Theorem~\ref{thm:main}]
	We let $\xi$ be a point in $X$ such that $k(\xi,\xi)>1$ and $\varphi$ be  as in \eqref{eq:initial}. It follows from Lemma~\ref{lem:powers} that
	\[M\varphi^m(\xi) = \varphi^m(\xi) = \kappa(\xi,\xi)^{m/2},\]
	which we recall is bounded above by $1$. We fix $0<a<1$ and let $\varphi_a$ be as in \eqref{eq:phiaexp}. Appealing again to Lemma~\ref{lem:powers}, we find that
	\[M \varphi_a(\xi) = a + (1-a^2) \sum_{m=1}^\infty a^{m-1} \kappa(\xi,\xi)^{m/2} =  a + \frac{(1-a^2) \sqrt{\kappa(\xi,\xi)}}{1-a \sqrt{\kappa(\xi,\xi)}}. \]
	Since $\|\varphi_a\|_{\mathcal{M}_k}\leq1$ by Lemma~\ref{lem:conformal} and Lemma~\ref{lem:hkappa}, we see that $\xi$ cannot belong to the Bohr set $\mathcal{B}(\mathcal{M}_k)$ if $M \varphi_a(\xi) > 1$ for some $a$ in $(0,1)$. This is equivalent to
	\[\sqrt{\kappa(\xi,\xi)} > \frac{1}{1+2a} \qquad \iff \qquad k(\xi,\xi) > \frac{(1+2a)^2}{4a(1+a)}.\]
	The right-hand side decreases to $9/8$ as $a \to 1^-$. Hence, if $k(\xi,\xi)>9/8$, then $\xi$ does not belong to $\mathcal{B}(\mathcal{M}_k)$.
\end{proof}

\section{The Drury--Arveson space} \label{sec:da}
We again begin in the classical setting, this time recalling the Bohr--Wiener proof that every $z$ in $\mathbb{D}$ with $\lvert z \rvert \leq 1/3$ belongs to the Bohr set of $H^\infty$. The key point is that if $\varphi(z) = \sum_{m\geq0}a_mz^m$ satisfies $\|\varphi\|_{H^\infty}\leq1$, then 
\begin{equation} \label{eq:fwienerclassical}
	\lvert a_m \rvert \leq 1 - \lvert a_0 \rvert^2,
\end{equation}
for $m=1,2,3,\ldots$. It follows that if $\lvert z \rvert \leq 1/3$, then 
\[M\varphi(z) \leq \lvert a_0 \rvert + (1-\lvert a_0\rvert^2) \sum_{m=1}^\infty \lvert z \rvert^m \leq \lvert a_0 \rvert + \frac{1-\lvert a_0 \rvert^2}{2} \leq 1.\]
The key inequality \eqref{eq:fwienerclassical} is typically proved by recognizing that the case $m=1$ can be obtained from the Schwarz lemma and then extending to the case $m\geq2$ by the \emph{F.~Wiener trick}. For expositional reasons we will give a somewhat different proof, where we use F.~Wiener's trick first, followed by manipulations involving conformal maps.

\begin{proof}[Proof of \eqref{eq:fwienerclassical}]
	Fix $m\geq1$. If $\lvert a_0 \rvert = 1$, then $\varphi(z) \equiv a_0$ and there is nothing to prove, so assume that $\lvert a_0 \rvert < 1$. Let $\omega$ be a primitive $m$th root of unity and replace $\varphi$ by the function
	\[\frac{1}{m} \sum_{j=0}^{m-1} \varphi(\omega^j z) = \sum_{n=0}^\infty a_{mn} z^{mn},\]
	which plainly maps $\mathbb{D}$ to itself by the triangle inequality. It is equally plain that
	\[\psi(z) \coloneq \frac{a_0-\varphi(z)}{1-\overline{a_0}\varphi(z)}\]
	satisfies $\|\psi\|_{H^\infty} \leq 1$ and that $\psi$ vanishes to order at least $m$ at the origin. Inverting, we get
	\[\varphi(z) = a_0 - (1-\lvert a_0 \rvert^2) \sum_{n=1}^\infty \overline{a_0}^{\,n-1} \psi^n(z),\]
	where $\psi^n$ vanishes to order at least $mn$ at the origin. Only the term $n=1$ therefore contributes to the coefficient of $z^m$, whence $a_m = -(1-\lvert a_0\rvert^2) \widehat{\psi}(m)$. Parseval's identity gives $\lvert \widehat{\psi}(m) \rvert \leq \|\psi\|_{H^2} \leq \|\psi\|_{H^\infty} \leq 1$, which completes the proof.
\end{proof}

Two features of this argument must be replaced when we pass to the Drury--Arveson space $H^2_d$: the individual Taylor coefficients give way to the homogeneous parts of a multiplier, and the rotations $z \mapsto \omega^j z$ to the corresponding self-maps of $\mathbb{B}_d$. The argument is otherwise the same, with one exception: here it is plain that $\|\psi\|_{H^\infty}\leq 1$, but for the corresponding statement in the Drury--Arveson setting we will use Lemma~\ref{lem:conformal}.

Recall that the $m$-homogeneous part of $F(z) = \sum_\gamma a_\gamma z^\gamma$ in $H^2_d$ is 
\[F_m(z) = \sum_{\lvert \gamma\rvert = m} a_\gamma z^\gamma.\]
The decomposition $F = \sum_{m\geq0} F_m$ is orthogonal, so in particular it holds that $\|F_m\|_{H^2_d} \leq \|F\|_{H^2_d}$.

Let $\theta$ be a real number. It is plain that the operator 
\[U_\theta F(z) \coloneq F(e^{i\theta}z)\] 
is unitary on $H^2_d$ and that $(U_\theta F)_m = e^{i\theta m} F_m$. Since $U_\theta(FG) = (U_\theta F)(U_\theta G)$ whenever $F$, $G$, and $FG$ are in $H^2_d$, it follows that $M_{U_\theta\Phi} = U_\theta M_\Phi U_\theta^{-1}$ for every multiplier $\Phi$, whence
\begin{equation} \label{eq:gauge}
	\|U_\theta\Phi\|_{\mult(H^2_d)} = \|\Phi\|_{\mult(H^2_d)}.
\end{equation}

We are now in a position to establish F.~Wiener's inequality \eqref{eq:fwienerclassical} for the Drury--Arveson space. Since the monomials are orthonormal in $H^2_1 = H^2$, the following statement reduces to \eqref{eq:fwienerclassical} when $d=1$.

\begin{lemma} \label{lem:fwienerda}
	If $\|\Phi\|_{\mult(H^2_d)} \leq 1$, then
	\[\|\Phi_m\|_{H^2_d} \leq 1 - \lvert \Phi(0) \rvert^2\]
	for $m=1,2,3,\ldots$.
\end{lemma}

\begin{proof}
	Set $a_0 = \Phi(0)$. If $\lvert a_0 \rvert = 1$, then 
	\[1 \geq \|\Phi\|_{\mult(H^2_d)}^2 \geq \|\Phi\|_{H^2_d}^2 = \lvert a_0\rvert^2 + \sum_{m\geq1}\|\Phi_m\|_{H^2_d}^2\]
	forces $\Phi \equiv a_0$ and there is nothing to prove, so assume that $\lvert a_0\rvert<1$. Fix $m\geq1$. Let $\omega$ be a primitive $m$th root of unity and replace $\Phi$ by
	\[\frac{1}{m} \sum_{j=0}^{m-1} U_{2\pi j/m}\Phi,\]
	which, by \eqref{eq:gauge} and the triangle inequality, is again a multiplier of norm at most $1$, has the same value at the origin and the same $m$-homogeneous part, and satisfies $\Phi(\omega z) = \Phi(z)$. By Lemma~\ref{lem:conformal}, the function
	\[\Psi(z) \coloneq \frac{a_0-\Phi(z)}{1-\overline{a_0}\Phi(z)}\]
	satisfies $\|\Psi\|_{\mult(H^2_d)} \leq 1$ and $\Psi(0)=0$. Since also $\Psi(\omega z) = \Psi(z)$, every nonzero homogeneous part of $\Psi$ has degree at least $m$. Inverting, we get
	\[\Phi = a_0 - (1-\lvert a_0 \rvert^2) \sum_{n=1}^\infty \overline{a_0}^{\,n-1} \Psi^n,\]
	where every nonzero homogeneous part of $\Psi^n$ has degree at least $mn$. Only the term $n=1$ therefore contributes to the $m$-homogeneous part, whence $\Phi_m = -(1-\lvert a_0\rvert^2)\Psi_m$. Since the homogeneous parts of $\Psi$ are mutually orthogonal in $H^2_d$, we get $\|\Psi_m\|_{H^2_d} \leq \|\Psi\|_{H^2_d} \leq \|\Psi\|_{\mult(H^2_d)} \leq 1$, which completes the proof.
\end{proof}

To pass from the homogeneous parts of a multiplier to its coefficients, we need the norms of the monomials. Writing $\gamma! = \prod_{j=1}^d \gamma_j!$, we get from the multinomial theorem that
\begin{equation} \label{eq:multinom}
	\langle z,w\rangle_d^m = \sum_{\lvert \gamma \rvert = m} \frac{m!}{\gamma!}\, z^\gamma \overline{w^\gamma},
\end{equation}
so expanding $k_d(z,w) = \sum_{m\geq0} \langle z,w\rangle_d^m$, we find that $\|z^\gamma\|_{H^2_d}^2 = \gamma!/\lvert\gamma\rvert!$.

\begin{proof}[Proof of Theorem~\ref{thm:da}]
	Since $\kappa(z,w) = \langle z,w\rangle_d$, we have $k_d(z,z)>9/8$ if and only if $\|z\|_{\ell^2}>1/3$, so part (i) of Theorem~\ref{thm:main} applied to $H^2_d$ shows that no such $z$ belongs to the Bohr set of $\mult(H^2_d)$. Suppose next that $\Phi(z) = \sum_\gamma a_\gamma z^\gamma$ satisfies $\|\Phi\|_{\mult(H^2_d)}\leq1$. By the Cauchy--Schwarz inequality and \eqref{eq:multinom}, we obtain that
	\[\sum_{\lvert \gamma\rvert=m} \lvert a_\gamma z^\gamma\rvert \leq \left(\sum_{\lvert\gamma\rvert=m} \lvert a_\gamma\rvert^2 \frac{\gamma!}{m!}\right)^{\frac12} \left(\sum_{\lvert\gamma\rvert=m}\frac{m!}{\gamma!}\lvert z^\gamma\rvert^2\right)^{\frac12} = \|\Phi_m\|_{H^2_d}\, \|z\|_{\ell^2}^m.\]
	Summing over $m$ and using Lemma~\ref{lem:fwienerda}, we conclude as in the classical case that if $\|z\|_{\ell^2} \leq 1/3$, then 
	\[M\Phi(z) \leq \lvert a_0\rvert + (1-\lvert a_0\rvert^2) \sum_{m=1}^\infty \|z\|_{\ell^2}^m \leq \lvert a_0\rvert + \frac{1-\lvert a_0\rvert^2}{2} \leq 1. \qedhere\]
\end{proof}

Let us next turn to the second part of the proof of Theorem~\ref{thm:main}, which uses Theorem~\ref{thm:da}. Let $\mathcal{H}_k$ be an admissible reproducing kernel Hilbert space on $X$ with the diagonal complete Pick property and let $c_\lambda \geq 0$ be the coefficients of $\kappa$. Set
\[S \coloneq \{\lambda \neq 0 \,:\, c_\lambda > 0\}\]
and let $d$ denote the cardinality of $S$, which is at least $1$ because $k \not\equiv 1$, and enumerate $S = \{\lambda_j\}_{j=1}^d$. Define $\mathbf{b}_j(x) \coloneq \sqrt{c_{\lambda_j}} e_{\lambda_j}(x)$ and set
\[\mathbf{b}(x) \coloneq (\mathbf{b}_j(x))_{j=1}^d.\]
The main idea is that $\langle \mathbf{b}(x), \mathbf{b}(y) \rangle_d = \kappa(x,y)$. Since $\kappa(x,x)<1$, it follows that $\mathbf{b}$ maps $X$ into $\mathbb{B}_d$ with $\mathbf{b}(x_0)=0$. It also yields the key identity
\begin{equation} \label{eq:kdb}
	k(x,y) = \frac{1}{1-\langle \mathbf{b}(x),\mathbf{b}(y)\rangle_d} = k_d\left(\mathbf{b}(x),\mathbf{b}(y)\right).
\end{equation}
The relevance of \eqref{eq:kdb} is the following result of Agler and McCarthy, which is the precise form of the universality of the Drury--Arveson space that we require.

\begin{lemma} \label{lem:universality}
	Fix $d$ in $\mathbb{N} \cup \{\infty\}$, a set $X$, and a map $\mathbf{b} \colon X \to \mathbb{B}_d$. Let $\mathcal{H}_k$ be the reproducing kernel Hilbert space on $X$ whose kernel is
	\[k(x,y) = k_d\left(\mathbf{b}(x),\mathbf{b}(y)\right).\]
	Then $F \mapsto F \circ \mathbf{b}$ is a co-isometry from $H^2_d$ onto $\mathcal{H}_k$. If moreover $\varphi$ is in $\mathcal{M}_k$, then there is $\Phi$ in $\mult(H^2_d)$ such that
	\[\varphi = \Phi \circ \mathbf{b} \qquad\text{and}\qquad \|\Phi\|_{\mult(H^2_d)} = \|\varphi\|_{\mathcal{M}_k}.\]
\end{lemma}

Lemma~\ref{lem:universality} is \cite{Hartz2023}*{Theorem~10.4.19}, which is stated there for a normalized complete Pick space and the map obtained by Kolmogorov factorization of $1-1/k$. The proof requires nothing of $\mathbf{b}$ beyond the displayed relation, which in our case is \eqref{eq:kdb}.

\begin{proof}[Second part of the proof of Theorem~\ref{thm:main}]
	Let $\varphi$ be in $\mathcal{M}_k$ with $\|\varphi\|_{\mathcal{M}_k}\leq1$. By \eqref{eq:kdb} and Lemma~\ref{lem:universality}, there is $\Phi(z) = \sum_\gamma a_\gamma z^\gamma$ in $\mult(H^2_d)$ with $\varphi = \Phi \circ \mathbf{b}$ and $\|\Phi\|_{\mult(H^2_d)}\leq1$.
	
	In order to use Theorem~\ref{thm:da}, we need to understand the relationship between the majorants of $\varphi$ and $\Phi$. To that end, we fix a finitely supported multi-index $\gamma=(\gamma_j)_{j=1}^d$ of nonnegative integers and use that $e_\lambda e_\mu = e_{\lambda+\mu}$ to compute
	\begin{equation} \label{eq:bmonomial}
		\mathbf{b}(x)^\gamma = \prod_{j=1}^d \mathbf{b}_j(x)^{\gamma_j} = \left(\prod_{j=1}^d c_{\lambda_j}^{\gamma_j/2}\right) e_{\iota(\gamma)}(x),
	\end{equation}
	where $\iota(\gamma) \coloneq \sum_{j=1}^d \gamma_j \lambda_j$. This shows that every monomial in the coordinates of $\mathbf{b}$ is a nonnegative multiple of a single basis function. The monomials form an orthogonal basis of $H^2_d$, so the expansion of $\Phi$ converges in $H^2_d$, and since composition with $\mathbf{b}$ is a co-isometry by Lemma~\ref{lem:universality}, it follows that $\varphi = \sum_\gamma a_\gamma \mathbf{b}^\gamma$ converges in $\mathcal{H}_k$. Taking the inner product with $e_\lambda$ and appealing to \eqref{eq:bmonomial}, we get
	\[\widehat{\varphi}(\lambda) = \sum_{\iota(\gamma) = \lambda} a_\gamma \prod_{j=1}^d c_{\lambda_j}^{\gamma_j/2},\]
	whence
	\begin{equation} \label{eq:transfer}
		M\varphi(x) \leq \sum_\gamma \left\lvert a_\gamma\, \mathbf{b}(x)^\gamma\right\rvert = M\Phi\left(\mathbf{b}(x)\right)
	\end{equation}
	for every $x$ in $X$, by the triangle inequality.

	Suppose finally that $k(x,x)\leq9/8$, so that $\|\mathbf{b}(x)\|_{\ell^2} = \sqrt{\kappa(x,x)} \leq 1/3$. Theorem~\ref{thm:da} gives $M\Phi(\mathbf{b}(x)) \leq \|\Phi\|_{\mult(H^2_d)} \leq 1$, and \eqref{eq:transfer} completes the proof.
\end{proof}

Notice that \eqref{eq:transfer} is in general a strict inequality, since distinct $\gamma$ with $\iota(\gamma)=\lambda$ may contribute to $\widehat{\varphi}(\lambda)$ with cancellation.

It is worth isolating what the diagonal complete Pick property contributed. Lemma~\ref{lem:universality} applies to any factorization $\langle \mathbf{b}(x),\mathbf{b}(y)\rangle_d = \kappa(x,y)$, and such a factorization is available as soon as $\kappa$ is positive semi-definite, by Kolmogorov's theorem. What the requirement $c_\lambda \geq 0$ adds is that the factorization can be carried out in the orthogonal basis $\{e_\lambda\}_{\lambda \in \Lambda}$ itself. Each coordinate of $\mathbf{b}$ is then a multiple of a basis function, whence \eqref{eq:bmonomial} and the transfer \eqref{eq:transfer}. For an arbitrary factorization, the coordinates of $\mathbf{b}$ bear no relation to $\{e_\lambda\}_{\lambda\in\Lambda}$, and there is no reason for the majorant to transfer at all.

\section{Corollary~\ref{cor:powerseries} and Corollary~\ref{cor:diriseries}} \label{sec:fj}
This final section establishes Corollary~\ref{cor:powerseries} and Corollary~\ref{cor:diriseries}.

\begin{proof}[Proof of Corollary~\ref{cor:powerseries}]
	As explained in the introduction, $\mathcal{H}_k$ has the diagonal complete Pick property, so Theorem~\ref{thm:main} applies. Since $k(z,z) = \sum_{j\geq0} b_j \lvert z \rvert^{2j}$ and $r \mapsto k(r,r)$ is strictly increasing on $[0,1)$ with $k(0,0)=1$, the equation $k(\varrho,\varrho)=9/8$ has at most one solution in $[0,1)$. If it has one, then 
	\[\mathcal{B}(\mathcal{M}_k) = \{z \in \mathbb{D}\,:\, \lvert z \rvert \leq \varrho\}.\] Otherwise $k(r,r)<9/8$ for every $0 \leq r<1$ and $\mathcal{B}(\mathcal{M}_k)=\mathbb{D}$.

	It remains to prove that $\varrho \geq 1/3$, with equality only for the Szeg\H{o} kernel. Write $\kappa(z,w) = \sum_{j\geq1} c_j (\overline{w}z)^j$ with $c_j \geq 0$. Since $\kappa(r,r)<1$ for $0\leq r<1$, monotone convergence gives $\sum_{j\geq1}c_j\leq1$ and therefore
	\begin{equation} \label{eq:universaldisc}
		\kappa(r,r) = \sum_{j=1}^\infty c_j r^{2j} \leq r^2 \sum_{j=1}^\infty c_j \leq r^2,
	\end{equation}
	that is, $\sqrt{\kappa(r,r)} \leq r$. It follows from \eqref{eq:famform} that every $z$ with $\lvert z \rvert \leq 1/3$ belongs to $\mathcal{B}(\mathcal{M}_k)$, so that $\varrho \geq 1/3$. If $\varrho=1/3$, then both inequalities in \eqref{eq:universaldisc} are equalities at $r=1/3$. The first forces $c_j = 0$ for every $j\geq2$, and the second then forces $c_1=1$, so that $\kappa(z,w)=\overline{w}z$ and $k$ is the Szeg\H{o} kernel.
\end{proof}

The proof of Corollary~\ref{cor:diriseries} goes along the same lines as that of Corollary~\ref{cor:powerseries}, and we omit it. Let us only point out that the critical abscissa comes out as $\sigmaa+\log3/\log2$ because the computation \eqref{eq:universaldisc} gives $\sqrt{\kappa(\sigma,\sigma)} \leq 2^{-(\sigma-\sigmaa)}$ in place of $\sqrt{\kappa(r,r)}\leq r$, and $2^{-(\sigma-\sigmaa)} \leq 1/3$ precisely when $\sigma \geq \sigmaa+\log3/\log2$.

\bibliography{bohrmultcnp}

\end{document}